\documentclass[12pt]{article}
\usepackage{amsmath,amssymb,amsthm,mathtools, xcolor,kotex}
\usepackage[noadjust]{cite}
\usepackage{tikz}
\usepackage{hyperref}
\usetikzlibrary{arrows.meta,positioning,shapes.geometric,calc} 
\newtheorem{theorem}{Theorem}[section]
\newtheorem{lemma}[theorem]{Lemma}
\newtheorem{proposition}[theorem]{Proposition}
\newtheorem{corollary}[theorem]{Corollary}
\newtheorem{claim}[theorem]{Claim}
\theoremstyle{definition}
\newtheorem{definition}[theorem]{Definition} 
\newtheorem{observation}[theorem]{Observation}
\theoremstyle{remark}
\newtheorem{remark}[theorem]{Remark}
 
\DeclareMathOperator{\ONH}{ONH}
\DeclareMathOperator{\per}{per} 
\newcommand{\HH}{\mathbb{H}}
\newcommand{\CC}{\mathcal{C}}
\newcommand{\HG}{\mathcal{H}}
\usepackage{geometry}
\definecolor{revcol}{RGB}{200,0,150}

\title{Det-extremal cubic graphs and the total domatic number} 
 
\author{Myungho Choi\thanks{Department of Mathematics Education, Seoul National University, Seoul,
Republic of Korea.}
\and
Hyemin Kwon\thanks{Korea Institute for Advanced Study (KIAS), Seoul, Republic of Korea.}
\and
Boram Park\thanks{Department of Mathematics Education, Seoul National University, Seoul,
Republic of Korea. \texttt{borampark@snu.ac.kr}}
}
 
\date{}
 
\begin{document}
\maketitle

\begin{abstract}
A graph $G$ is \emph{det-extremal} if $|\det A|=\per A$ for its adjacency matrix $A$.
Det-extremal cubic bipartite graphs arise in the study of P\'olya's permanent problem, and McCuaig characterized the $3$-connected ones as vertex-sums of copies of the Heawood graph.
The \emph{total domatic number} of a graph is the largest number of pairwise disjoint total dominating sets. 
Characterization of the cubic graphs with  total domatic number $1$ has been a long-standing open problem.

In this paper, we prove that a connected cubic graph is det-extremal if and only if its total domatic number is $1$. We further show that McCuaig's characterization extends to all $3$-connected cubic graphs, and that every connected det-extremal cubic graph has girth $3$, $5$ or $6$.  We also prove that a connected det-extremal cubic non-bipartite graph has at least $28$ vertices, and that this bound is best possible.  
Through this correspondence, these results carry over to cubic graphs with total domatic number $1$.
In addition, in the language of configurations, our results imply that every triangle-free $3$-configuration has a blocking set. 
\end{abstract}

\medskip
\noindent\textbf{Keywords:} cubic graph, det-extremal graph, total domatic number, $3$-configuration

\smallskip

\noindent\textbf{2020 MSC:} 05C50, 05C69, 05B30

\section{Introduction}\label{sec:intro}

In this paper, all graphs are finite and simple, unless stated otherwise.
Hypergraphs are also finite, but multiple hyperedges are allowed.
In a graph or a hypergraph $G$, for a vertex $v$ in $G$, let $\deg_G(v)$ and $N_G(v)$ be the degree of $v$ and the neighborhood of $v$, respectively, in $G$. 
For a graph or a hypergraph $G$ and a subset $X$ of $V(G)$ or $E(G)$, let $G-X$ denote the graph or the hypergraph obtained from $G$ by deleting the members of $X$.
If $X$ is a subset of $V(G)$, then $G-X$ also deletes all edges incident with a vertex in $X$.
If $X=\{x\}$, then we may write $G-x$ instead of $G-X$.
 
\subsection{Det-extremal graphs}\label{sub:intro:det} 
 
For a square $(0,1)$-matrix $M$, it holds that $|\det M|\le\per M$. We say that a graph $G$ is \emph{det-extremal} if $|\det A|=\per A$ for its adjacency matrix $A$.
For bipartite graphs, det-extremality has been studied as a special case of P\'olya's permanent problem \cite{Polya}, which asks for which square $(0,1)$-matrices $M$ some entries can be changed from $1$ to $-1$ so that the determinant of the resulting matrix equals $\per M$. 
Little \cite{Little} characterized such matrices in terms of a family of forbidden subgraphs, and a structural characterization yielding a polynomial algorithm was obtained independently by McCuaig \cite{McCuaigPolya} and by Robertson, Seymour, and Thomas \cite{RST}. 
Thomassen \cite{Thom86} related the problem to sign-nonsingular matrices and to directed cycles of even length.
For cubic bipartite graphs, McCuaig \cite{McCuaig} characterized the $3$-connected det-extremal graphs, and Funk, Jackson, Labbate, and Sheehan \cite{FJLS} studied the det-extremal graphs of connectivity $2$ and determined the possible orders of connected det-extremal cubic bipartite graphs.
 
A bipartite graph $G$ is \emph{Pfaffian} if some matrix $A^{*}$ obtained from an adjacency matrix $A$ of $G$ by changing some entries from $1$ to $-1$ satisfies $|\det A^{*}|=\per A$.
Then P\'olya's permanent problem asks which bipartite graphs are Pfaffian, and the forbidden subgraph and structural descriptions of \cite{Little,McCuaigPolya,RST} mentioned above are characterizations of Pfaffian bipartite graphs.
Every det-extremal bipartite graph with parts of equal size is Pfaffian, while the converse fails, see Remark~\ref{rem:pfaffian}. 

To state the characterization of McCuaig \cite{McCuaig}, we need the notion of a vertex-sum of two cubic bipartite graphs.
Let $G_1,G_2$ be disjoint cubic bipartite graphs, let $x\in V(G_1)$ with $N_{G_1}(x)=\{x_1,x_2,x_3\}$, and let $y\in V(G_2)$ with $N_{G_2}(y)=\{y_1,y_2,y_3\}$.
A \emph{vertex-sum} of $G_1$ and $G_2$ with respect to $x$ and $y$ is a graph $G$ obtained from $(G_1-x)\cup (G_2-y)$ by adding edges $x_1y_1,\ x_2y_2,\ x_3y_3$.
See Figure~\ref{fig:vsum}. 
Note that the resulting graph $G$ is also cubic bipartite and $|V(G)|=|V(G_1)|+|V(G_2)|-2$.
We often omit reference to $x$ and $y$, and simply call the resulting graph $G$ a vertex-sum of $G_1$ and $G_2$ if there is no confusion.
For a positive integer $k$ and a connected graph $G$, a \emph{$k$-vertex-sum of $G$} is a graph obtained from the disjoint union of $k$ copies of $G$ by repeatedly replacing two connected components by one of their vertex-sums, until the graph is connected. 
In particular, $G$ itself is a $1$-vertex-sum of $G$. 
The following theorem of McCuaig \cite{McCuaig} characterizes the $3$-connected det-extremal cubic bipartite graphs.
The Heawood graph, denoted by $\HH$ in this paper,  is given in Figure~\ref{fig:heawood}.
 
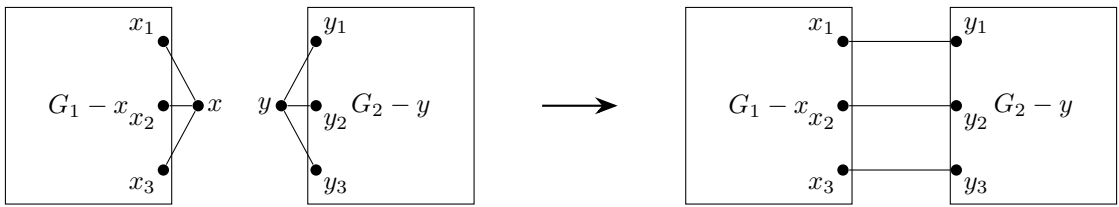
\begin{figure}[ht]
\centering
\begin{tikzpicture}[
 v/.style={circle,fill,inner sep=1.5pt},
 blob/.style={draw,rectangle,minimum width=22mm,minimum height=26mm,align=center,font=\footnotesize},
 every label/.style={font=\footnotesize,inner sep=1pt}]
\node[blob] (B1) at (0,0) {\footnotesize{$G_1-x$}};
\node[v,label=above left:$x_1$] (x1) at (0.99,0.85) {};
\node[v,label=below left:$x_2$] (x2) at (0.99,0) {};
\node[v,label=below left:$x_3$] (x3) at (0.99,-0.85) {};
\node[v,label=right:$x$] (y) at (1.45,0) {};
\draw (y)--(x1); \draw (y)--(x2); \draw (y)--(x3);
\node[v,label=left:$y$] (x) at (2.55,0) {};
\node[v,label=above right:$y_1$] (y1) at (3.01,0.85) {};
\node[v,label=below right:$y_2$] (y2) at (3.01,0) {};
\node[v,label=below right:$y_3$] (y3) at (3.01,-0.85) {};
\node[blob] (B2) at (4,0) {\footnotesize{$G_2-y$}};
\draw (x)--(y1); \draw (x)--(y2); \draw (x)--(y3);
\draw[-{Stealth[length=3mm]},thick] (6,0) -- (7,0);
\node[blob] (C1) at (9.0,0) {$G_1-x$};
\node[v,label=above left:$x_1$] (a1) at (9.98,0.85) {};
\node[v,label=below left:$x_2$] (a2) at (9.98,0) {};
\node[v,label=below left:$x_3$] (a3) at (9.98,-0.85) {};
\node[v,label=above right:$y_1$] (b1) at (11.48,0.85) {};
\node[v,label=below right:$y_2$] (b2) at (11.48,0) {};
\node[v,label=below right:$y_3$] (b3) at (11.48,-0.85) {};
\node[blob] (C2) at (12.5,0) {$G_2-y$};
\draw (a1)--(b1); \draw (a2)--(b2); \draw (a3)--(b3);
\end{tikzpicture}
\caption{An illustration of a vertex-sum}
\label{fig:vsum}
\end{figure}

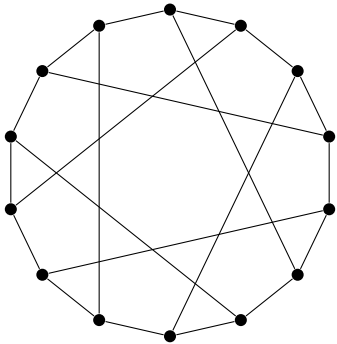
\begin{figure}[b!]
\centering
\begin{tikzpicture}[scale=0.9,
 v/.style={circle,fill,inner sep=1.6pt},
 every label/.style={font=\scriptsize,inner sep=1.5pt}]
\foreach \i in {0,...,13}{
 \node[v] (h\i) at ({90-\i*360/14}:2.4) {};
}
\foreach \i in {0,...,13}{
 \pgfmathtruncatemacro{\j}{mod(\i+1,14)}
 \draw (h\i) -- (h\j);
}
\foreach \i in {0,2,4,6,8,10,12}{
 \pgfmathtruncatemacro{\j}{mod(\i+5,14)}
 \draw (h\i) -- (h\j);
}
\end{tikzpicture}
\caption{The Heawood graph $\HH$}
\label{fig:heawood}
\end{figure}

\begin{theorem}[\cite{McCuaig}, see also {\cite[Theorem~2.2]{FJLS}}]
\label{thm:mccuaig}
A $3$-connected cubic bipartite graph is det-extremal if and only if it is a $k$-vertex-sum of the Heawood graph $\HH$ for some integer $k\ge 1$.
\end{theorem} 
 
For general connectivity, Funk, Jackson, Labbate, and Sheehan \cite{FJLS} determined the possible orders as follows. 
 
\begin{theorem}[{\cite[Theorem~5.1]{FJLS}}]\label{thm:fjlsorders}
There exists a connected det-extremal cubic bipartite graph of order $n$ if and only if $n\in\{14, 26, 38, 42, 44, 50\}$ or even $n\ge 54$. 
\end{theorem}
 
In this paper, we study det-extremal cubic graphs in general, and one of our main results identifies them with the cubic graphs of total domatic number $1$.

\subsection{The total domatic number}\label{sub:intro:dt} 
 
Let $G$ be a graph with $\delta(G)\ge 1$.
A \emph{total dominating set} of $G$ is a set $S\subseteq V(G)$ with $S\cap N_G(v)\ne\emptyset$ for every $v\in V(G)$, and the \emph{total domatic number} $d_t(G)$ is the largest number of pairwise disjoint total dominating sets.
The parameter was introduced by Cockayne, Dawes, and Hedetniemi \cite{CDH}.
For a positive integer $k$, a \emph{$k$-coloring} of a graph $G$ is an assignment from $V(G)$ to a set of $k$ colors.
Chen, Kim, Tait, and Verstra\"ete \cite{CKTV} observed that $d_t(G)\ge k$ if and only if $G$ admits a \emph{$k$-coupon coloring}, that is, a $k$-coloring of $G$ in which every vertex has all $k$ colors in its neighborhood.
Goddard and Henning \cite{GH} studied the same notion under the name \emph{thoroughly dispersed coloring}. 
 
In particular, $d_t(G)\ge2$ if and only if $G$ has a $2$-coupon coloring.
Heggernes and Telle \cite{HT} proved that deciding whether $d_t(G)\ge2$ is NP-complete and Zelinka \cite{Zelinka} showed that no lower bound on the minimum degree guarantees $d_t(G)\ge2$.
 
For $r$-regular graphs the total domatic number can be large: it is at least $r/(3\log r)$ \cite{ASV}, and in fact at least $(1-o(1))\,r/\log r$ \cite{CKTV}. 
Asking only for $d_t(G)\ge2$ for an $r$-regular graph $G$ is therefore a weak demand, and yet it is not always met.
If $r=1$, it is easy to see that $d_t(G)=1$.
The case $r=2$ follows from the fact that $d_t(C_n)\ge2$ if and only if $4\mid n$.
For $r\ge4$, we always have $d_t(G)\ge2$: the case $r\ge8$ was proved by Alon and Bregman \cite{AB}, and the case $r\ge4$ follows from a theorem of Thomassen \cite{Thom92} on even cycles in regular digraphs, as shown by Henning and Yeo \cite{HY13}.
For $r=3$, the answer can be no.
The Heawood graph $\HH$ has $d_t(\HH)=1$, and infinitely many further cubic examples exist \cite{DHH}.
Thus the problem of deciding whether $d_t(G)\ge 2$ is settled for regular graphs $G$ except for the cubic case.

Indeed, characterizing the cubic graphs with $d_t(G)=1$ was posed as an open problem in \cite{AMMY} and is described as a long-standing open problem in \cite{GoddardHenning}. 
The total domatic number of a cubic graph has therefore attracted attention, and it is also closely related to notions appearing in other contexts, in the language of colorings and of configurations.  
 
Note that every cubic graph $G$ satisfies $d_t(G) \le 3$, and that $d_t(G) = 3$ holds if and only if $G$ admits a $3$-coloring in which every vertex has exactly one neighbor of each color; such graphs are called \emph{neighborhood $3$-balanced}. 
Minyard and Sepanski \cite{MS} gave a complete characterization of the cubic graphs $G$ with $d_t(G) = 3$. 
It thus remains to decide, for a cubic graph $G$, whether $d_t(G) = 2$ or $d_t(G) = 1$.
 
The known sufficient conditions for a cubic graph $G$ to satisfy $d_t(G)\ge2$ all force a short cycle.
Desormeaux, Haynes, and Henning \cite{DHH} proved $d_t(G)\ge2$ when $G$ contains a diamond, when $G$ has no induced cycle of length greater than $5$, and when $G$ is claw-free.
They asked whether a triangle suffices.
Akbari, Motiei, Mozaffari, and Yazdanbod \cite{AMMY} proved it under a condition forcing every vertex to lie on a triangle or a $4$-cycle.
Most recently, Akbari, Azimian, Fazli Khani, Samimi, and Zahiri \cite{AAKSZ} answered the question of Desormeaux, Haynes, and Henning \cite{DHH} negatively, exhibiting a cubic graph $G$ of order $60$ containing a triangle with $d_t(G)=1$. 
They then proved the following theorem, which implies the case of a diamond. 
 
\begin{theorem}[{\cite[Corollary~4]{AAKSZ}}]\label{AAKSZ-thm}
Every connected cubic graph $G$ containing a $4$-cycle has $d_t(G)\ge2$.
\end{theorem}

Planarity is another sufficient condition. Goddard and Henning~\cite{GH} conjectured that every planar triangulation with minimum degree at least $3$ has total domatic number at least $2$.
This was proved by Francis, Illickan, Jose, and Rajendraprasad~\cite{FIJR24}, who conjectured more generally that the same holds for every planar graph with minimum degree at least~3.
Rotenberg, Rutschmann, and Thomassen~\cite{RRT26} recently confirmed this conjecture.
In particular, $d_t(G)\ge 2$ for every planar cubic graph~$G$.

We characterize the connected cubic graphs with total domatic number $1$, and from this result, one can determine whether a given cubic graph $G$ satisfies $d_t(G)\ge 2$.

\subsection{Main results}\label{sub:intro:results}
 
We now state our results.
For cubic bipartite graphs, the correspondence between det-extremality and the total domatic number was already implicit in earlier work \cite{McCuaig,Seymour,Thom86}, in the language of configurations or of directed cycles of even length, see Section~\ref{sec:B:configuration} for a precise comparison.
Our first result extends this correspondence for all connected cubic graphs.

\begin{theorem}\label{thm:detper}
Let $G$ be a connected cubic graph. Then $G$ is det-extremal if and only if $d_t(G)=1$.
\end{theorem}
 
For the adjacency matrix $A$ of the Heawood graph $\HH$, it can be checked that $|\det A|=\per A=576$ and so $d_t(\HH)=1$ by Theorem~\ref{thm:detper}.
As a by-product we obtain a new proof of Theorem~\ref{AAKSZ-thm} (see Corollary~\ref{cor:4cycle}).

A related approach to Theorem~\ref{thm:detper} was used by Rotenberg, Rutschmann, and Thomassen~\cite{RRT26} in the context of planar graphs, and  Theorem~\ref{thm:detper} makes the resulting correspondence explicit for all connected cubic graphs. See Subsection~\ref{sec:proofmain}.

Under the additional assumption of $3$-connectedness we obtain a complete description of the structure of det-extremal cubic graphs. Namely, the bipartiteness assumption in Theorem~\ref{thm:mccuaig} can be dropped.
 
\begin{theorem}\label{thm:3connected}
Let $G$ be a $3$-connected cubic graph.
Then $G$ is det-extremal if and only if $G$ is a $k$-vertex-sum of the Heawood graph for some integer $k\ge 1$.
\end{theorem}

From the properties of a $k$-vertex-sum of the Heawood graph, 
we obtain three conditions, each easy to check, under which a $3$-connected cubic graph $G$ is not det-extremal. 
 
\begin{proposition}\label{thm:three:conditions}
A $3$-connected cubic graph $G$ is not det-extremal if at least one of the following holds.
\textup{(1)} $G$ is non-bipartite.
\textup{(2)} The girth of $G$ is not $6$. 
\textup{(3)} $|V(G)|\not\equiv2\pmod{12}$.
\end{proposition}

Our next result concerns the girth.
 
\begin{theorem}\label{thm:girth}
Let $G$ be a connected det-extremal cubic graph.
Then the girth of $G$ is $3$, $5$ or $6$. 
\end{theorem}

All three girths permitted by Theorem~\ref{thm:girth} are known to occur: girths $3$ and $5$ are realized by examples in \cite{AAKSZ}
(stated there in the language of the total domatic number),
and girth $6$ is realized by the Heawood graph.
Hence Theorem~\ref{thm:girth} is best possible.

From Theorem~\ref{thm:fjlsorders}, we know that a smallest connected det-extremal cubic bipartite graph has $14$ vertices.
However, as far as we know, no non-bipartite example of order less than $60$ was previously known, see \cite{AAKSZ} for an example of order $60$.
It follows from our results that the smallest such example has at least $28$ vertices, and we construct one with exactly $28$ vertices (see Figure~\ref{fig:28vertex}).  

\begin{theorem}\label{thm:no22}
Every connected det-extremal cubic non-bipartite graph has at least $28$ vertices.
Moreover, there is a connected det-extremal cubic non-bipartite graph with $28$ vertices.  
\end{theorem}

By the correspondence of Theorem~\ref{thm:detper},
Theorems~\ref{thm:3connected}, \ref{thm:girth}, and \ref{thm:no22},
as well as Proposition~\ref{thm:three:conditions},
can be restated in terms of the total domatic number.

\begin{corollary}\label{cor:nonbip2}
Let $G$ be a connected cubic graph.
\begin{itemize}
\item[\rm (i)] When $d_t(G)=1$, $G$ has girth $3$, $5$, or $6$.
\item[\rm (ii)] When $G$ is $3$-connected, $d_t(G)\ge 2$ if at least one of the following holds.
\textup{(1)} $G$ is non-bipartite.
\textup{(2)} The girth of $G$ is not $6$. 
\textup{(3)} $|V(G)|\not\equiv2\pmod{12}$.
\item[\rm (iii)] When $d_t(G)=1$ and $G$ is non-bipartite, $G$ has at least $28$ vertices.
\end{itemize}
\end{corollary}

As noted in Subsection~\ref{sub:intro:dt}, the known non-bipartite examples of cubic graphs with total domatic number $1$ in the literature contain a triangle or a $5$-cycle, and Corollary~\ref{cor:nonbip2}(i)  shows that this is not a coincidence.

The first item (1) of Corollary~\ref{cor:nonbip2}(ii) answers the question of Desormeaux, Haynes, and Henning \cite{DHH}, whether a triangle forces $d_t(G)\ge2$ for a connected cubic graph $G$, in the affirmative for $3$-connected cubic graphs. 
The last item (3) of Corollary~\ref{cor:nonbip2}(ii) shows that a simple counting of vertices often suffices, whereas deciding $d_t(G)\ge2$ is NP-complete for a graph $G$ \cite{HT}.

The bound in Corollary~\ref{cor:nonbip2}(iii) is sharp as there exists a connected cubic non-bipartite graph $G$ of order $28$ with $d_t(G)=1$.
Thus the minimum order of such a graph is $28$.

\medskip

Our results also have a consequence for configurations.
A \emph{$3$-configuration} is a $3$-uniform $3$-regular hypergraph in which two distinct vertices lie in at most one common hyperedge, and a \emph{blocking set} of a configuration is a set of vertices meeting every hyperedge and containing no hyperedge.
The existence of blocking sets in $3$-configurations has been investigated since the 1990s \cite{DGropp,HGropp91,Gropp,EGS21}, and the question of which orders admit a blocking set free $3$-configuration was posed three times at the British Combinatorial Conference \cite{BCC1,BCC2,BCC3} before being settled by Funk, Jackson, Labbate, and Sheehan \cite{FJLS}.
In Section~\ref{sec:B:configuration} we deduce from our results that every triangle-free $3$-configuration has a blocking set (Corollary~\ref{cor:triangle}), whereas configurations of small order were verified by computer enumeration.

The paper is organized as follows.
Section~\ref{sec:prelim} collects the results we quote and develops some basic properties of bipartite double covers.
Section~\ref{sec:hyper} establishes the main tool and then proves Theorem~\ref{thm:detper}.
Section~\ref{sec:heawood} studies vertex-sums of the Heawood graph and proves Theorem~\ref{thm:3connected} and Proposition~\ref{thm:three:conditions}. 
Section~\ref{sec:cubic} proves Theorems~\ref{thm:girth} and \ref{thm:no22}.
Section~\ref{sec:B:configuration} interprets our results in the language of configurations.
 
\section{Preliminaries}\label{sec:prelim}

For a bipartite graph with parts of equal size, det-extremality is defined via a biadjacency matrix in the literature, as we now record. 
Let $G$ be a bipartite graph with parts $X$ and $Y$.
A \emph{biadjacency matrix} $B$ of $G$ is the $(0,1)$-matrix $B=(b_{xy})_{x\in X,\,y\in Y}$, rows indexed by $X$ and columns by $Y$, with $b_{xy}=1$ if and only if $xy\in E(G)$.
It is the off-diagonal block of the adjacency matrix $A$, that is, $A=\left(\begin{smallmatrix}O &B\\B^{\mathsf T}& O\end{smallmatrix}\right)$.
If $|X|=|Y|$, then $\per B$ is the number of perfect matchings of $G$, $\per A=(\per B)^2$, and $|\det A|=|\det B|^2$.
Thus $|\det A|=\per A$ if and only if $|\det B|=\per B$.
We use whichever of $A$ and $B$ is convenient when we speak of the det-extremality (or Pfaffian) of bipartite graphs with parts of equal size.

\subsection{Basic observations}
The following are folklore, but we include the proofs for completeness.
 
\begin{lemma}\label{lem:parity}
Let $G$ be a cubic graph. 
\begin{itemize}
\item[\rm (1)] If $G-e$ is bipartite for some $e\in E(G)$, then $G$ is bipartite.
\item[\rm (2)] If $G$ is connected and bipartite, then its connectivity is $2$ or $3$.
\end{itemize}
\end{lemma}
 
\begin{proof} 
Let $H$ be a subcubic bipartite graph with parts $X$ and $Y$ such that $\delta(H)\ge 2$, and let $d$ be the number of vertices of degree $2$ in $H$. Suppose that every vertex of degree $2$ is in $X$.
Then counting the edges of $H$ from each side gives $3|X|-d=|E(H)|=3|Y|$, so $3$ divides $d$.
 
For (1), suppose $G-e$ is bipartite with $e=uv$.
If $u$ and $v$ lie in the same part, then applying the count to $H=G-e$ with $d=2$ gives a contradiction.
Thus $u$ and $v$ lie in different parts and so $G$ is bipartite.
For (2), recall that in a cubic graph the vertex connectivity equals the edge connectivity, so it suffices to show that $G$ has no bridge.
Let $e=uv$ be a bridge of $G$ and let $H$ be the connected component of $G-e$ containing $u$, with the bipartition inherited from $G$.
The count with $d=1$ again gives a contradiction.
\end{proof}

Here is a simple lemma on det-extremal graphs.
 
\begin{lemma}\label{lem:coverdetM}
If a graph $G$ has a perfect matching, then $G$ is det-extremal if and only if every connected component of $G$ is det-extremal.
\end{lemma}
 
\begin{proof}
Let $A$ be the adjacency matrix of $G$.
Let $G_1,\dots,G_t$ be the connected components of $G$ and let $A_i$ be the adjacency matrix of $G_i$ for each $i\in\{1,\ldots,t\}$. 
Ordering the vertices component by component makes $A$ block diagonal with diagonal blocks $A_1,\dots,A_t$, and both determinant and permanent are multiplicative over the blocks.
Since $|\det A_i|\le\per A_i$ for each $i$, it follows that if every connected component of $G$ is det-extremal then so is $G$, and that the converse holds whenever $\per A\ne0$.
A perfect matching of $G$ gives a permutation contributing $1$ to $\per A$, so $\per A\ne0$.
\end{proof}
 
A cycle $C$ of a bipartite graph $G$ is \emph{central} if $G-V(C)$ has a perfect matching.
It is well-known that a $4$-cycle $C$ of a cubic bipartite graph $G$ is central.
Letting $H=G-V(C)$, with parts $X$ and $Y$ inherited from $G$, for every $S\subseteq X$,  
$3|N_H(S)|\ge 3|S|-2$ and thus $|N_H(S)|\ge |S|$.
By Hall's theorem, $H$ has a perfect matching.
Combined with the following lemma, this yields Observation~\ref{lem:noc4}.
 
\begin{lemma}[{\cite[Lemma~1.1(b)]{FJLS}}]\label{thm:central4} 
A bipartite graph is det-extremal if and only if every central cycle has length congruent to $2$ modulo $4$.
\end{lemma}
 
\begin{observation}\label{lem:noc4}
A cubic bipartite graph containing a $4$-cycle is not det-extremal.
\end{observation}

\subsection{Bipartite double covers}\label{sub:doublecover}

The \emph{bipartite double cover} of a graph $G$, also called its \emph{Kronecker double cover}, is the graph $G\times K_2$ with vertex set $V(G)\times\{0,1\}$ in which $(u,i)$ and $(v,1-i)$ are adjacent for every $i\in\{0,1\}$ whenever $uv\in E(G)$.
Equivalently, $G\times K_2$ is the bipartite graph whose biadjacency matrix is the adjacency matrix of $G$.
 
\begin{lemma}\label{lem:coverprop}
Let $G$ be a connected graph. Then the following hold. 
\begin{itemize}
\item[\rm(1)] $G$ is det-extremal if and only if $G\times K_2$ is det-extremal.
\item[\rm(2)] If $G$ has girth at least $7$, then $G\times K_2$ has girth at least $8$.
\item[\rm(3)] If $G$ is non-bipartite, then $G\times K_2$ is connected.
\item[\rm(4)] If $G$ is $3$-connected, cubic, and non-bipartite, then $G\times K_2$ is cubic and $3$-connected.
\end{itemize}
\end{lemma}
 
\begin{proof} Let $G'=G\times K_2$.
 
(1) Let $A$ be the adjacency matrix of $G$.
By the definition, $G'$ is the bipartite graph with biadjacency matrix $A$.
Ordering its two sides, and using that $A$ is symmetric, its adjacency matrix is $A'=\left(\begin{smallmatrix}0&A\\A&0\end{smallmatrix}\right)$.
Then $\det A'=(-1)^{|V(G)|}(\det A)^{2}$ and $\per A'=(\per A)^{2}$, so $|\det A'|=|\det A|^{2}$.
Both $|\det A|$ and $\per A$ are nonnegative reals, so $|\det A|^{2}=(\per A)^{2}$ if and only if $|\det A|=\per A$.
 
(2) Since $G'$ is bipartite, its girth is even.
Suppose $G'$ has a cycle $C'$ of length $L\in\{4,6\}$, and let $C$ be the projection of $C'$ to $G$, a closed walk in $G$.
Note that two vertices of $C'$ have the same projection precisely when they are the two lifts of one vertex of $G$.
If $C$ has no repeated vertex, then $C$ is a cycle of $G$ of length at most $6$, a contradiction.
Suppose that $C$ repeats a vertex $v$. 
Then $C'$ contains both lifts of $v$.
These two vertices divide $C'$ into two paths, whose projections are closed walks at $v$ in $G$ of odd lengths $\ell_1,\ell_2$ with $\ell_1+\ell_2=L\le 6$.
The shorter, say $\ell_1$, is odd and at most $3$.
It is not $1$, since $G$ has no loop, so $\ell_1=3$ and $G$ has a triangle, a contradiction.
 
(3) Note that $G$ has a closed walk of odd length through any vertex $v$, and it lifts to a walk from $(v,0)$ to $(v,1)$, so the two lifts of every vertex lie in one connected component.
Thus $G'$ is connected.
 
(4) By (3) and the definition of $G'$, $G'$ is connected and cubic. 
Since $G'$ is cubic, its vertex connectivity equals its edge connectivity, so it is enough to show $G'$ has no edge cut of size at most $2$.
Suppose to the contrary that $G'$ has an edge cut $F'$ with $|F'|\le2$.  
Let $F$ be the set of edges of $G$ having a lift in $F'$, so that $1\le|F|\le2$.
Clearly, $G'-F'$ contains $(G-F)\times K_2$ as a spanning subgraph.
As $G$ is $3$-edge-connected, $G-F$ is connected.
 
If $G-F$ is non-bipartite, then 
$(G-F)\times K_2$ is connected by (3) and hence so is $G'-F'$, a contradiction.
Suppose now that $G-F$ is bipartite, with parts $P$ and $Q$.
Then $(G-F)\times K_2$ has exactly two connected components, namely the subgraphs of $(G-F)\times K_2$ induced by $(P\times\{0\})\cup(Q\times\{1\})$ and by $(P\times\{1\})\cup(Q\times\{0\})$. 
We call them $C'_0$ and $C'_1$.
Each of $C'_0$ and $C'_1$ is isomorphic to $G-F$.
As $G$ is non-bipartite, Lemma~\ref{lem:parity}(1) rules out $|F|=1$, so $|F|=2$ and $F'$ consists of exactly one lift of each edge of $F$.
Again as $G$ is non-bipartite and $G-F$ is bipartite, some $e\in F$ has both ends in $P$ or both ends in $Q$, and then both lifts of $e$ join $C'_0$ to $C'_1$.
One of them lies outside $F'$, so $G'-F'$ is connected, a contradiction.
\end{proof}

\section{Hypergraphs and proof of Theorem~\ref{thm:detper}}\label{sec:hyper}
 
Subsection~\ref{sub:hyper} proves the main tool for our results (Theorem~\ref{thm:thomassen}).
Subsection~\ref{sec:proofmain} proves Theorem~\ref{thm:detper}.

\subsection{Non-$2$-colorable hypergraphs}\label{sub:hyper}
 
A hypergraph is \emph{$2$-colorable} (or \emph{bipartite}) if its vertices can be $2$-colored with no monochromatic hyperedge.
A \emph{subhypergraph} of a hypergraph $\HG$ is one obtained from $\HG$ by deleting hyperedges, and it is \emph{proper} if at least one hyperedge is deleted.
A hypergraph is \emph{minimally non-$2$-colorable} if it is not $2$-colorable but every proper subhypergraph of it is.
For a positive integer $r$, a hypergraph is \emph{$r$-uniform} if all hyperedges have the same cardinality $r$, and \emph{$r$-regular} if every vertex lies in exactly $r$ hyperedges.
An \emph{incidence matrix} of a hypergraph, the rows indexed by the vertices and the columns by the hyperedges, is a $(0,1)$-matrix that records which vertex lies in which hyperedge.
Its \emph{incidence graph} is the bipartite graph whose biadjacency matrix is the incidence matrix.
 
We shall use the following result of Henning and Yeo \cite{HY13}.
 
\begin{lemma}[{\cite[Corollary~1]{HY13}}]\label{lem:hy}
Every connected $3$-uniform $3$-regular hypergraph is either $2$-colorable, or becomes so on deleting any one of its hyperedges.
\end{lemma}

The results of Vazirani and Yannakakis \cite{VY} and of Seymour \cite{Seymour} are quoted below, and we record here why their statements are about det-extremality.
Let $M=(m_{ij})$ be a $(0,1)$-matrix of order $n$, and let $D(M)$ be the digraph on $\{1,\dots,n\}$ with an arc from $i$ to $j$ whenever $i\ne j$ and $m_{ij}=1$.
 
\begin{lemma}[{\cite[Lemma~2.2]{VY}}]\label{lem:evencycle}
Let $M$ be a $(0,1)$-matrix of order $n$ all of whose diagonal entries are $1$.
Then $\det M=\per M$ if and only if $D(M)$ has no directed cycle of even length.
\end{lemma}
 
We call an incidence matrix all of whose diagonal entries are $1$ a \emph{diagonal-incidence} matrix.
 
\begin{proposition}\label{prop:three}
Let $\HG$ be a $3$-uniform $3$-regular hypergraph.
Then $\HG$ has a diagonal-incidence matrix $M$ and for any such $M$ the following three conditions are equivalent.
\begin{itemize}
\item[\rm(1)] The incidence graph of $\HG$ is det-extremal.
\item[\rm(2)] $|\det M|=\per M$.
\item[\rm(3)] $D(M)$ has no directed cycle of even length.
\end{itemize}
\end{proposition}
 
\begin{proof}
A $3$-uniform $3$-regular hypergraph has as many hyperedges as vertices, so an incidence matrix of $\HG$ is square and its incidence graph $G$ is cubic bipartite.
Then $G$ has a perfect matching by Hall's theorem.
Listing the vertices and hyperedges of $\HG$ along this perfect matching puts the matching on the diagonal.
So $\HG$ has a diagonal-incidence matrix, and we let $M$ be one.
The matrix $M$ is a biadjacency matrix of the incidence graph of $\HG$, so (1)$\Leftrightarrow$(2).
For a $(0,1)$-matrix the nonzero terms in the expansion of $\det M$ are the signs of the permutations counted by $\per M$, and the identity permutation has sign $+1$.
So $|\det M|=\per M$ forces $\det M=\per M$, and (2)$\Leftrightarrow$(3) by Lemma~\ref{lem:evencycle}.
\end{proof}
 
By Proposition~\ref{prop:three}, a $3$-uniform $3$-regular hypergraph can be read in three languages, those of graphs, of matrices, and of digraphs.
We now state Seymour's result in the form in which we use it.
A digraph is \emph{strong} (or \emph{strongly connected}) if for every two vertices $u$ and $v$ there is a directed path from $u$ to $v$.
 
\begin{theorem}[{\cite[Corollary to Proposition~3]{Seymour}}]\label{thm:seymour}
Let $\HG$ be a hypergraph with pairwise distinct hyperedges and as many hyperedges as vertices, and let $M$ be a diagonal-incidence matrix of $\HG$. 
Then $\HG$ is minimally non-$2$-colorable if and only if $D(M)$ is strong and has no directed cycle of even length.
\end{theorem}

In what follows, we use Theorem~\ref{thm:thomassen} rather than Theorem~\ref{thm:seymour} itself, since the hypergraphs used to prove our main theorems may have multiple hyperedges, whereas Theorem~\ref{thm:seymour} and the classical statement for configurations (see Section~\ref{sec:B:configuration}) require the hyperedges to be pairwise distinct.

\begin{theorem}\label{thm:thomassen}
A connected $3$-uniform $3$-regular hypergraph is non-$2$-colorable if and only if its incidence graph is det-extremal.
\end{theorem}
 
\begin{proof}
Let $\HG$ be a connected $3$-uniform $3$-regular hypergraph, and let $M$ be a diagonal-incidence matrix of $\HG$, which exists by Proposition~\ref{prop:three}.
 
First, suppose that two hyperedges $e$ and $f$ of $\HG$ are equal as sets.
By Lemma~\ref{lem:hy}, either $\HG$ is $2$-colorable or $\HG-e$ is, and in the latter case the hyperedge $f$ remains in $\HG-e$, so a $2$-coloring of $\HG-e$ is one of $\HG$.
In either case, $\HG$ is $2$-colorable.
The two columns of $M$ corresponding to $e$ and $f$ are equal, which with the diagonal entries $1$ gives a directed cycle of $D(M)$ of length $2$, so the incidence graph is not det-extremal by Proposition~\ref{prop:three}.
Both sides of the equivalence fail, so the equivalence holds in this case.
 
Now assume that the hyperedges of $\HG$ are pairwise distinct.
Since $\HG$ is $3$-uniform and $3$-regular, the hyperedges are as many as the vertices of $\HG$.
Thus we can apply Theorem~\ref{thm:seymour} to $M$.
By Proposition~\ref{prop:three} it suffices to prove that $\HG$ is non-$2$-colorable if and only if $D(M)$ has no directed cycle of even length.
 
Suppose that $\HG$ is not $2$-colorable.
By Lemma~\ref{lem:hy}, the hypergraph $\HG-e$ is $2$-colorable for every hyperedge $e$, and every proper subhypergraph of $\HG$ is contained in some $\HG-e$.
Thus $\HG$ is minimally non-$2$-colorable.
By Theorem~\ref{thm:seymour}, $D(M)$ has no directed cycle of even length.
Conversely, suppose that $D(M)$ has no directed cycle of even length.
We check that $D(M)$ is strong.
Each row or column of $M$ has three entries equal to $1$, as $\HG$ is $3$-uniform and $3$-regular.
The diagonal entries being $1$, every vertex of $D(M)$ therefore has in-degree and out-degree $2$, and so $D(M)$ is a union of Eulerian digraphs.
The $j$th column of $M$ consists of $j$ together with the in-neighbors of $j$ in $D(M)$, so the hyperedges of $\HG$ are the sets $\{j\}\cup N^{-}_{D(M)}(j)$ and the underlying graph of $D(M)$ has the same number of connected components as $\HG$.
Since $\HG$ is connected, it follows that $D(M)$ is strongly connected.
By Theorem~\ref{thm:seymour}, $\HG$ is minimally non-$2$-colorable.
In particular $\HG$ is not $2$-colorable. 
\end{proof}
 
In Theorem~\ref{thm:thomassen}, connectedness is essential.
If $\HG$ is the disjoint union of the Fano plane (the projective plane of order $2$, whose incidence graph is the Heawood graph $\HH$) and the hypergraph on three points whose three hyperedges all equal the point set, then $\HG$ is a $3$-uniform $3$-regular hypergraph that is non-$2$-colorable, but its incidence graph is not det-extremal. 
Theorem~\ref{thm:thomassen} will be applied to connected hypergraphs only, and the following lemma handles the disconnected case.
 
The \emph{dual} $\HG^{*}$ of a hypergraph $\HG$ is obtained by interchanging the roles of the vertices and hyperedges.
The incidence graph of $\HG^{*}$ is that of $\HG$ with the two sides of the bipartition interchanged, hence the same graph.
 
\begin{lemma}\label{lem:disconnected}
Let $\HG$ be a $3$-uniform $3$-regular hypergraph whose connected components are $\HG_1$ and $\HG_2$ with $\HG_2\cong \HG_1^{*}$. 
Then $\HG$ is non-$2$-colorable if and only if its incidence graph is det-extremal.  
\end{lemma}

\begin{proof}
An incidence graph $G$ of $\HG$ has exactly two connected components $G_1$ and $G_2$, where $G_i$ is an incidence graph of $\HG_i$ for each $i$.
Note that each $\HG_i$ is connected, $3$-uniform, and $3$-regular.
Moreover, since a connected $3$-uniform $3$-regular hypergraph and its dual have the same incidence graph, $G_1$ is isomorphic to $G_2$. 
By Theorem~\ref{thm:thomassen}, the following equivalence holds.
\begin{align*}
 \text{$\HG_1$ is non-$2$-colorable} \Leftrightarrow G_1\text{ is det-extremal} 
 \Leftrightarrow G_2\text{ is det-extremal} 
 \Leftrightarrow \text{$\HG_2$ is non-$2$-colorable}. 
\end{align*}
Thus, $\HG$ is non-$2$-colorable if and only if $\HG_i$ is non-$2$-colorable for each $i$ if and only if $G_i$ is det-extremal for each $i$.
Since $G$ is cubic bipartite and so has a perfect matching, it follows from Lemma~\ref{lem:coverdetM} that $G$ is det-extremal if and only if $G_i$ is det-extremal for each $i$.
Therefore the statement holds.
\end{proof}
 
\subsection{Proof of Theorem~\ref{thm:detper}}\label{sec:proofmain}
 
For a graph $G$ with $\delta(G)\ge 1$, the \emph{open neighborhood hypergraph} of $G$, denoted $\ONH(G)$, is the hypergraph with vertex set $V(G)$ whose hyperedges consist of the open neighborhoods $N_G(v)$ of the vertices $v$ of $G$, where multiple hyperedges are possible.
If $G$ is $r$-regular, then $\ONH(G)$ is an $r$-uniform $r$-regular hypergraph.
It follows immediately from the definitions that for a graph $G$ with $\delta(G)\ge 1$, $d_t(G)\ge2$ if and only if $\ONH(G)$ is $2$-colorable. 
 
\begin{observation}\label{prop:levi}
Let $G$ be a connected cubic graph with adjacency matrix $A$.
Then the following hold.
\begin{itemize}
\item[\rm (1)] With the hyperedge $N_G(v)$ indexed by $v$ for each $v\in V(G)$, an incidence matrix of $\ONH(G)$ is $A$, so the incidence graph of $\ONH(G)$ is the bipartite double cover $G\times K_2$.
\item[\rm (2)] If $G$ is bipartite, then $\ONH(G)$ has exactly two connected components and they are dual to each other.
\end{itemize}
\end{observation}
 
Observation~\ref{prop:levi}(1) and (2) are immediate from the definitions.
Now we are ready to give a proof of Theorem~\ref{thm:detper}.
 
\begin{proof}[Proof of Theorem~\ref{thm:detper}]
Let $G$ be a connected cubic graph.
Then $\ONH(G)$ is a $3$-uniform $3$-regular hypergraph whose incidence graph is $G\times K_2$ by Observation~\ref{prop:levi}(1).
If $G$ is non-bipartite, then $\ONH(G)$ is connected by Lemma~\ref{lem:coverprop}(3) and Observation~\ref{prop:levi}(1), and so we apply Theorem~\ref{thm:thomassen} to $\ONH(G)$.
If $G$ is bipartite, then $\ONH(G)$ has two connected components dual to each other by Observation~\ref{prop:levi}(2), and so Lemma~\ref{lem:disconnected} applies.
Either way,
\begin{align*}
 \text{$G$ is det-extremal} & \Leftrightarrow \text{$G\times K_2$ is det-extremal}
 & \text{by Lemma~\ref{lem:coverprop}(1),} \\
 & \Leftrightarrow \text{$\ONH(G)$ is not $2$-colorable}
 & \text{by Theorem~\ref{thm:thomassen} and Lemma~\ref{lem:disconnected},}\\
 & \Leftrightarrow  d_t(G)=1
 & \text{by the definitions}.
\end{align*}
\end{proof}

A similar approach appears in~\cite{RRT26}, where Rotenberg, Rutschmann, and Thomassen pass to $G\times K_2$ and apply theorems of McCuaig~\cite{McCuaig,McCuaigPolya} in their study of planar graphs, although the connection to det-extremality is not stated there.
Our proof instead relies on Theorem~\ref{thm:thomassen}, which is based on Seymour's theorem~\cite{Seymour} and holds for all connected 3-uniform 3-regular hypergraphs.

The following is a new proof of Theorem~\ref{AAKSZ-thm}, which now comes as a corollary of Theorem~\ref{thm:detper}.
 
\begin{corollary}\label{cor:4cycle}
Every connected cubic graph $G$ containing a $4$-cycle has $d_t(G)\ge2$. 
\end{corollary}
 
\begin{proof}
Let $uvwxu$ be a $4$-cycle of $G$.
It lifts to the $4$-cycle $(u,0)(v,1)(w,0)(x,1)(u,0)$ of the cubic bipartite graph $G\times K_2$, which is therefore not det-extremal by Observation~\ref{lem:noc4}.
Neither is $G$, according to Lemma~\ref{lem:coverprop}(1).
By Theorem~\ref{thm:detper}, $d_t(G)\ne1$, and so $d_t(G)\ge 2$.
\end{proof}

\section{Heawood graphs and proof of Theorem~\ref{thm:3connected}}\label{sec:heawood}

We begin by recording from the definition that a vertex-sum of two cubic bipartite graphs is cubic bipartite.
The $3$-connectedness is also preserved. 
 
\begin{observation}[{\cite[Lemma~31]{McCuaig}}]\label{obs:vsum3conn}
A vertex-sum of two $3$-connected cubic bipartite graphs is again a $3$-connected cubic bipartite graph.
\end{observation}

The following lemma summarizes the properties of a $k$-vertex-sum of the Heawood graph for some integer $k\ge 1$.
 
\begin{lemma}\label{lem:Heawood_Property}
Let $G$ be a $k$-vertex-sum of the Heawood graph for some integer $k\ge 1$.
\begin{itemize}
\item[\rm (1)] $G$ is a $3$-connected cubic bipartite graph.
\item[\rm (2)] $G$ has $12k+2$ vertices.
\item[\rm (3)] If $k\ge2$, then $G$ contains two vertex-disjoint copies of the Heawood graph minus a vertex, each of which contains a $6$-cycle.
\item[\rm (4)] For every $e\in E(G)$, the graph $G-e$ contains a $6$-cycle.
\item[\rm (5)] $G$ has girth $6$.
\end{itemize}
\end{lemma} 
 
\begin{proof}
Suppose that $G$ is obtained from $k$ disjoint copies $H_1,\dots,H_k$ of the
Heawood graph $\HH$ by $k-1$ successive vertex-sums. 
We record the process obtaining $G$ as a graph $T_0$, as follows.
The vertex set of $T_0$ is $\{v_1,\dots,v_k\}$, and $v_i$ is joined to $v_j$ whenever the process deletes a vertex of $H_i$ on one side and a vertex of $H_j$ on the other.
Note that $T_0$ is a tree, since each of the $k-1$ vertex-sums merges two connected components into one.
 
It is clear that (1) holds by Observation~\ref{obs:vsum3conn}, since $\HH$ is $3$-connected.
It is easy to check that the graph $G$ has $14k-2(k-1)=12k+2$ vertices, and hence (2) holds.
We can also check that $\HH-v$ contains a $6$-cycle for every $v\in V(\HH)$ and that $\HH-e$ contains a $6$-cycle for every $e\in E(\HH)$; see Figure~\ref{fig:heawood}.

Suppose $k\ge2$.
Since the tree $T_0$ has at least two pendant vertices, $G$ contains two vertex-disjoint copies of $\HH-v$, each of which contains a $6$-cycle.
Hence (3) holds.
In addition, for every $e\in E(G)$, at least one of the two copies avoids $e$, and so (4) holds when $k\ge2$; when $k=1$, (4) holds since $\HH-e$ contains a $6$-cycle for every edge $e$.
 
Finally, we prove (5) by induction on $k$.
If $k=1$, then $G=\HH$ has girth $6$.
Suppose $k\ge 2$, and without loss of generality let $v_k$ be a pendant vertex of $T_0$.
Then $G$ is a vertex-sum of $G'$ and $H_k$ with respect to $x$ and $y$ for some $x\in V(G')$ and $y\in V(H_k)$, where $G'$ is a $(k-1)$-vertex-sum of the Heawood graph.
By the inductive hypothesis, $G'$ has girth $6$, and $H_k=\HH$ has girth $6$.
Suppose that $G$ contains a $4$-cycle $C$.
The subgraphs $G'-x$ and $H_k-y$ of $G'$ and $H_k$ have no $4$-cycle, so $C$ uses at least one of the three added edges. These three edges form an edge cut of $G$, so $C$ uses an even number of them, hence exactly two, say $x_1y_1$ and $x_2y_2$ with $x_1,x_2\in N_{G'}(x)$ and $y_1,y_2\in N_{H_k}(y)$, together with the edge $x_1x_2$ of $G'-x$ and the edge $y_1y_2$ of $H_k-y$.
Then $xx_1x_2$ is a triangle of $G'$, contradicting that $G'$ is bipartite.
Hence $G$ has no $4$-cycle.
Since $G$ is bipartite by (1), has no $4$-cycle, and contains a $6$-cycle by (4), the girth of $G$ is $6$.
\end{proof} 

Now we are ready to prove Theorem~\ref{thm:3connected}.
 
\begin{proof}[Proof of Theorem~\ref{thm:3connected}]
Let $G$ be a $3$-connected det-extremal cubic graph.
Suppose to the contrary that $G$ is non-bipartite.
Then $G\times K_2$ is det-extremal by Lemma~\ref{lem:coverprop}(1), and by Lemma~\ref{lem:coverprop}(3) and (4), $G\times K_2$ is $3$-connected.
Thus $G\times K_2$ is a $k$-vertex-sum of the Heawood graph for some integer $k\ge 1$ by Theorem~\ref{thm:mccuaig}, and hence has $12k+2$ vertices by Lemma~\ref{lem:Heawood_Property}(2).
So $|V( G\times K_2 )|=2|V(G)|=12k+2$ and so $|V(G)|=6k+1$ is odd, which is a contradiction since a cubic graph has an even number of vertices.
Hence $G$ is bipartite, and by Theorem~\ref{thm:mccuaig} $G$ is a $k$-vertex-sum of the Heawood graph for some integer $k\ge 1$.
 
Conversely, suppose that $G$ is a $k$-vertex-sum of the Heawood graph for some integer $k\ge 1$.
Such a graph is $3$-connected cubic bipartite by Lemma~\ref{lem:Heawood_Property}(1).
By Theorem~\ref{thm:mccuaig}, $G$ is det-extremal.
\end{proof}

\begin{proof}[Proof of Proposition~\ref{thm:three:conditions}]
We prove the contrapositive, so let $G$ be a $3$-connected det-extremal cubic graph.
By Theorem~\ref{thm:3connected}, $G$ is bipartite and is a $k$-vertex-sum of the Heawood graph for some integer $k\ge 1$.
In particular (1) fails.
By Lemma~\ref{lem:Heawood_Property}(5), the girth of $G$ is $6$, so (2) fails.
By Lemma~\ref{lem:Heawood_Property}(2), $|V(G)|=12k+2$, so (3) fails as well.
\end{proof}

\section{Girth, orders, and proofs of Theorems~\ref{thm:girth}~and~\ref{thm:no22}}\label{sec:cubic} 
 
Subsection~\ref{sub:fjls} recalls a decomposition of a $2$-connected cubic bipartite graph into $3$-connected pieces.
Subsection~\ref{sec:girth} uses it to prove  Theorem~\ref{thm:girth}.
Subsection~\ref{sub:nonbip} studies the orders of connected det-extremal cubic non-bipartite graph and proves Theorem~\ref{thm:no22}. 
 
\subsection{Decomposition into $3$-connected pieces}\label{sub:fjls}
 
We state the decomposition of a $2$-connected cubic bipartite graph into $3$-connected pieces,  described in Funk, Jackson, Labbate, and Sheehan \cite[Section~2]{FJLS}.
A \emph{cactus} is a connected multigraph each of whose blocks is a cycle.
 
\begin{definition}[{\cite[Section~2]{FJLS}}]\label{def:decompo}
Let $G$ be a $2$-connected cubic bipartite graph. 
Define $uRv$ if $u=v$ or $u$ and $v$ are joined by three edge-disjoint paths in $G$, let $V_1,\dots,V_t$ be the equivalence classes of $R$, and let $T$ be the multigraph obtained from $G$ by contracting each $V_i$ to a vertex $v_i$ and deleting loops. 
Then $T$ is a cactus, and we call $T$ the \emph{decomposition cactus} of $G$.
Let $H_i$ be the subgraph of $G$ induced by $V_i$.
For each $2$-edge cut $\{e,f\}$ of $T$ with $e$ and $f$ both incident with $v_i$, we add to $H_i$ the edge joining the two vertices of $V_i$ incident with $e$ and $f$ in $G$, which are distinct because $G$ is $2$-connected.
The resulting multigraph $G_i$ is called the \emph{$3$-connected piece at} $v_i$. 
The added edges are its \emph{marker edges}.
A \emph{$3$-bond} is the multigraph on two vertices joined by three parallel edges.
See Figure~\ref{fig:decomp} for an illustration.
\end{definition}
 
\begin{figure}[ht]
\centering
\begin{tikzpicture}[scale=0.9,
 v/.style={circle,fill,inner sep=1.3pt},
 cls/.style={draw,ellipse,minimum width=13mm,minimum height=13mm,font=\small},
 tn/.style={circle,draw,fill=white,inner sep=1.6pt},
 mk/.style={dashed,thick},
 capt/.style={font=\small,align=center}]
\begin{scope}[shift={(0,0)}]
 \node[cls] (C) at (0,0) {$V_2$};
 \node[cls] (U) at (90:2.05) {$V_1$};
 \node[cls] (L) at (210:2.05) {$V_3$};
 \node[cls] (R) at (330:2.05) {$V_4$};
 \foreach \n in {U,L,R}{\draw (C) to[bend left=16] (\n); \draw (C) to[bend right=16] (\n);}
 \node[capt] at (0,-3.05) {\footnotesize{(a) $G$, with the classes of $R$}\\ \footnotesize{and the six edges between them}};
\end{scope}
\begin{scope}[shift={(5.9,0)}]
 \node[tn,label=below:{\footnotesize$v_2$}] (c) at (0,0) {};
 \node[tn,label=above:{\footnotesize$v_1$}] (u) at (90:2.05) {};
 \node[tn,label=left:{\footnotesize$v_3$}] (l) at (210:2.05) {};
 \node[tn,label=right:{\footnotesize$v_4$}] (r) at (330:2.05) {};
 \foreach \n in {u,l,r}{\draw (c) to[bend left=16] (\n); \draw (c) to[bend right=16] (\n);}
 \node[capt] at (0,-3.05) {\footnotesize{(b) the decomposition cactus $T$} \\ \footnotesize{with three blocks}};
\end{scope}
\begin{scope}[shift={(11.8,0)}]
 \node[v] (p) at (-0.45,0.28) {}; \node[v] (q) at (0.45,0.28) {};
 \draw[mk] (p) to[bend left=38] (q); \draw[mk] (p) -- (q); \draw[mk] (p) to[bend right=38] (q);
 \node[font=\footnotesize] at (0,-0.28) {$G_2$, a $3$-bond};
 \foreach \a/\nm in {90/G_1, 210/G_3, 330/G_4}{
 \node[cls] (B\a) at (\a:2.15) {};
 \node[font=\small] at (\a:2.55) {$\nm$};
 \node[v] (s\a) at ($(\a:2.15)+(\a+90:0.45)$) {};
 \node[v] (t\a) at ($(\a:2.15)+(\a-90:0.45)$) {};
 \draw[mk] (s\a) -- (t\a);}
 \node[capt] at (0,-3.05) {\footnotesize{(c) the $3$-connected pieces,}\\ \footnotesize{marker edges dashed}};
\end{scope}
\end{tikzpicture}
\caption{An illustration of the decomposition in Definition~\ref{def:decompo}}
\label{fig:decomp}
\end{figure}
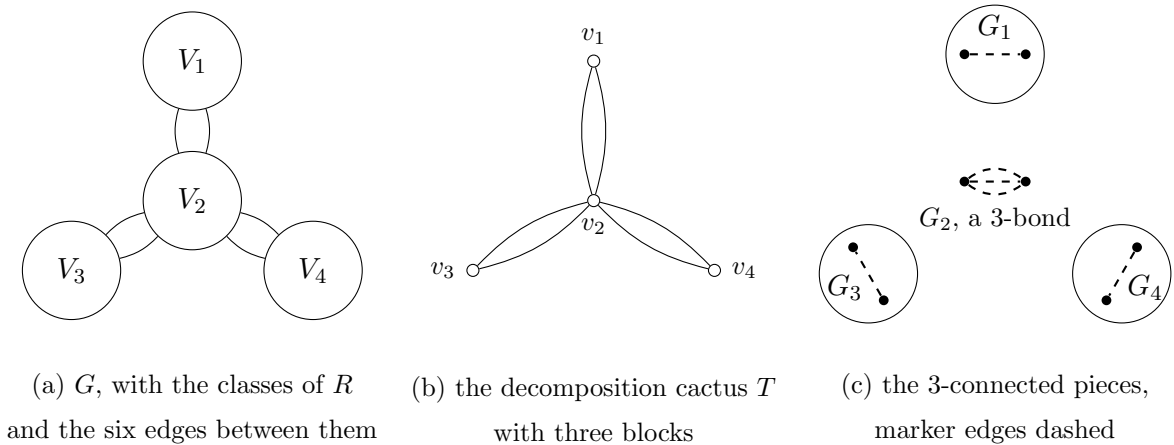
 
The following is immediate, and also stated in {\cite[Section~2]{FJLS}}.
 
\begin{observation}\label{obs:3con-decom}
In Definition~\ref{def:decompo}, each $G_i$ is either a $3$-bond or a simple $3$-connected cubic bipartite graph.
Thus each $G_i$ is $3$-edge-connected.
\end{observation}
 
\begin{lemma}\label{lem:cactus}
Let $G$ be a $2$-connected cubic bipartite graph with decomposition cactus $T$, and suppose that $T$ has at least two vertices. 
\begin{itemize}
\item[\rm (1)] For a vertex $v'\in V(T)$, if $\deg_T(v')=2$, then the $3$-connected piece $G'$ at $v'$ has exactly one marker edge.
\item[\rm (2)] $T$ has a vertex of degree exactly $2$.
\end{itemize}
\end{lemma}
 
\begin{proof}
(1) Since $\deg_T(v')=2$, there is exactly one pair of edges of $T$ incident to $v'$, so $G'$ has at most one marker edge.
Every block of $T$ is a cycle, so $T$ has no bridge.
As $T$ has at least two vertices, the two edges incident to $v'$ therefore form a $2$-edge cut of $T$, and $G'$ has a marker edge.
 
(2) If $T$ has a single block, that block is a cycle and every vertex has degree $2$.
Otherwise, $T$ has an end-block $B$ containing exactly one cut vertex of $T$.
As $B$ is a cycle it has at least two vertices, so some vertex of $B$ is not a cut vertex of $T$.
It lies in no other block and hence has degree $2$ in $T$.
\end{proof}
 
The last theorem we quote is the following.

\begin{theorem}[{\cite[Theorem~4.1]{FJLS}}]\label{thm:fjls41}
Let $G$ be a det-extremal cubic bipartite graph of connectivity $2$ and let $G_i$ be a $3$-connected piece at a vertex of degree at most $4$ in the decomposition cactus $T$ of $G$.
Then either $G_i$ is det-extremal, or $G_i$ is a $3$-bond at a vertex of degree exactly $4$ in $T$.
\end{theorem}
 
\subsection{Proof of Theorem~\ref{thm:girth}}\label{sec:girth}
 
Although the following lemma is intended for the proof of Theorem~\ref{thm:girth}, it is of independent interest, as its reformulation in the language of configurations in Section~\ref{sec:B:configuration} yields a new result.
 
\begin{lemma}\label{lem:detgirth}
Every connected det-extremal cubic bipartite graph has girth exactly $6$.
\end{lemma}
 
\begin{proof}
Let $G$ be a connected det-extremal cubic bipartite graph.
 Since $G$ is bipartite and, by Observation~\ref{lem:noc4}, has no $4$-cycle,  its girth is at least $6$.
It remains to produce a $6$-cycle.
By Lemma~\ref{lem:parity}(2), $G$ does not have a bridge, so its connectivity is $2$ or $3$.
If $G$ is $3$-connected, then by Theorem~\ref{thm:mccuaig} it is a $k$-vertex-sum of the Heawood graph for some integer $k\ge 1$, and by Lemma~\ref{lem:Heawood_Property}(5) it has girth $6$.
 
Suppose that $G$ has connectivity $2$.
Let $T$ be the decomposition cactus of $G$ and let $G_1,\dots,G_t$ be its $3$-connected pieces at $v_1,\ldots, v_t$, respectively, as in Definition~\ref{def:decompo}.
Here $t\ge2$, since $G$ has connectivity $2$.
By Lemma~\ref{lem:cactus}(2), choose $v_i$ with $\deg_T(v_i)=2$.
By Lemma~\ref{lem:cactus}(1), the $3$-connected piece $G_i$ at $v_i$ has exactly one marker edge, say $f$.
Hence $G_i$ is not a $3$-bond; indeed, since $G$ is simple, the two vertices of a $3$-bond piece are joined by at most one edge of $G$, so a $3$-bond piece has at least two marker edges. 
Then by Observation~\ref{obs:3con-decom} $G_i$ is a $3$-connected cubic bipartite graph.
By Theorem~\ref{thm:fjls41}, $G_i$ is det-extremal.
So $G_i$ is a $k$-vertex-sum of the Heawood graph for some integer $k\ge 1$ by Theorem~\ref{thm:mccuaig}, and $G_i-f$ contains a $6$-cycle by Lemma~\ref{lem:Heawood_Property}(4). 
As $G_i-f$ is a subgraph of $G$, $G$ contains a $6$-cycle.
\end{proof}
 
\begin{proof}[Proof of Theorem~\ref{thm:girth}]
Let $G$ be a connected det-extremal cubic graph.
By Lemma~\ref{lem:coverprop}(1), $G\times K_2$ is det-extremal.
Since a cubic bipartite graph has a perfect matching, every connected component of $G\times K_2$ is det-extremal by Lemma~\ref{lem:coverdetM}.
Each connected component of $G\times K_2$ is a connected det-extremal cubic bipartite graph, hence has girth $6$ by Lemma~\ref{lem:detgirth}.
Thus $G\times K_2$ has girth $6$.
By Lemma~\ref{lem:coverprop}(2), the girth of $G$ is at most $6$. 
By  Theorem~\ref{thm:detper} and Corollary~\ref{cor:4cycle} it has no $4$-cycle, so its girth lies in $\{3,5,6\}$.
\end{proof}

\begin{remark}\label{rem:pfaffian}
Recall from Section~\ref{sec:intro} that a bipartite graph with parts of equal size is Pfaffian if some matrix obtained from a biadjacency matrix $B$ by changing some entries from $1$ to $-1$ has absolute determinant equal to $\per B$.
Robertson, Seymour, and Thomas \cite{RST} and McCuaig \cite{McCuaigPolya} characterized connected cubic Pfaffian bipartite graphs (see also \cite[Theorem~2.1]{FJLS}).
Their characterization shows that infinitely many such graphs contain a $4$-cycle, whereas Lemma~\ref{lem:detgirth} says that a det-extremal cubic bipartite graph cannot. Thus, there exists a Pfaffian bipartite graph that is not det-extremal.
\end{remark}

\subsection{Proof of Theorem~\ref{thm:no22}}\label{sub:nonbip}

By Theorem~\ref{thm:fjlsorders}, a connected det-extremal cubic bipartite graph of order $n$ exists if and only if $n\in\{14, 26, 38, 42, 44, 50\}$ or even $n\ge 54$. 
Let $G$ be a connected det-extremal cubic non-bipartite graph of order $n$.
Then $G\times K_2$ is connected by Lemma~\ref{lem:coverprop}(3) and det-extremal by Lemma~\ref{lem:coverprop}(1) with $|V(G\times K_2)|=2n$.
Again by Theorem~\ref{thm:fjlsorders}, $2n\in\{14, 26, 38, 42, 44, 50\}$ or $2n\ge 54$.
Since $n$ is even, $n=22$ or $n\ge 28$.
Hence, every connected det-extremal cubic non-bipartite graph has $22$ vertices or at least $28$ vertices.

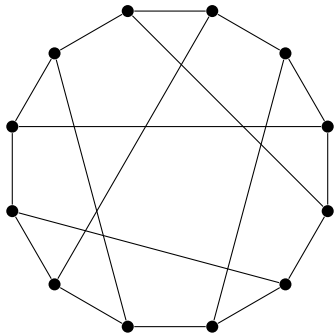
\begin{figure}[b!]
\centering
\begin{tikzpicture}[ scale=0.9,
  v/.style={circle,fill,inner sep=1.6pt}
]
\def\R{2.4}
\def\LX{-3.5}
\def\HX{3.5}
\foreach \i in {1,...,12}{
  \pgfmathsetmacro{\ang}{15+30*(\i-1)}
  \node[v] (t\i) at
    ({\LX+\R*cos(\ang)},{\R*sin(\ang)}) {};
}
\draw (t1)--(t2);
\draw (t2)--(t3);
\draw (t3)--(t4);
\draw (t4)--(t5);
\draw (t5)--(t6);
\draw (t6)--(t7);
\draw (t7)--(t8);
\draw (t8)--(t9);
\draw (t9)--(t10);
\draw (t10)--(t11);
\draw (t11)--(t12);
\draw (t1)--(t6);
\draw (t2)--(t10);
\draw (t3)--(t8);
\draw (t4)--(t12);
\draw (t5)--(t9);
\draw (t7)--(t11);
\draw (t12)--(t1);
 \end{tikzpicture}
\caption{The Twinplex graph $\mathbb{T}$}
\label{fig:twin}
\end{figure}

Here is a construction of a connected det-extremal cubic non-bipartite graph with $28$ vertices. The \emph{Twinplex graph} $\mathbb{T}$ given in Figure~\ref{fig:twin} is a cubic non-bipartite Hamiltonian graph on $12$ vertices, which is also denoted by $\Gamma_2$
in~\cite{FL}.  Choose the edge $e_T$ of $\mathbb{T}$ indicated in
Figure~\ref{fig:28vertex} and subdivide it once; denote the resulting graph by $\mathbb{T}^*$, and let $t$ be the new subdivision vertex. Similarly, choose an edge $e_H$ of the Heawood graph $\HH$ and subdivide it once; denote the resulting graph by $\HH^*$, and let $h$ be the new subdivision vertex. Finally, add the edge $th$, and the resulting graph $G$ is a cubic non-bipartite graph on $28$ vertices, see Figure~\ref{fig:28vertex}. 
Thus $G$ is obtained from the disjoint union of $\mathbb{T}^*$ and $\HH^*$ by adding the edge $th$. 

\begin{figure}[h!]
\centering
\begin{tikzpicture}[
  v/.style={circle,fill,inner sep=1.6pt}
]
\def\R{2.4}
\def\LX{-3.5}
\def\HX{3.5}
\foreach \i in {1,...,12}{
  \pgfmathsetmacro{\ang}{15+30*(\i-1)}
  \node[v] (t\i) at
    ({\LX+\R*cos(\ang)},{\R*sin(\ang)}) {};
}
\draw (t1)--(t2);
\draw (t2)--(t3);
\draw (t3)--(t4);
\draw (t4)--(t5);
\draw (t5)--(t6);
\draw (t6)--(t7);
\draw (t7)--(t8);
\draw (t8)--(t9);
\draw (t9)--(t10);
\draw (t10)--(t11);
\draw (t11)--(t12);
\draw (t1)--(t6);
\draw (t2)--(t10);
\draw (t3)--(t8);
\draw (t4)--(t12);
\draw (t5)--(t9);
\draw (t7)--(t11);
\node at (-0.4,0.75) {$e_T$};
\draw[->] (-0.6,0.6)--(-1.15,0.3); 
\node at (0.4,0.75) {$e_H$};
\draw[->] (0.6,0.6)--(1.15,0.3); 
\node at (-1.05,-0.18) {$t$};
\node at (1.0,-0.18) {$h$};
\node[v] at (-1.18,0) {};
\draw (t12)--(t1);
\foreach \i in {1,...,14}{
  \pgfmathsetmacro{\ang}{180+180/14-\i*360/14}
  \node[v] (h\i) at
    ({\HX+\R*cos(\ang)},{\R*sin(\ang)}) {};} 
\draw (h1)--(h2);
\draw (h2)--(h3);
\draw (h3)--(h4);
\draw (h4)--(h5);
\draw (h5)--(h6);
\draw (h6)--(h7);
\draw (h7)--(h8);
\draw (h8)--(h9);
\draw (h9)--(h10);
\draw (h10)--(h11);
\draw (h11)--(h12);
\draw (h12)--(h13);
\draw (h13)--(h14);
\draw (h1)--(h6);
\draw (h3)--(h8);
\draw (h5)--(h10);
\draw (h7)--(h12);
\draw (h9)--(h14);
\draw (h11)--(h2);
\draw (h13)--(h4);
\node[v] at (1.15,0) {};
\draw (h14)--(1.15,0)--(h1);
\draw (-1.15,0)--(1.15,0);
\end{tikzpicture}
\caption{A det-extremal cubic non-bipartite graph $G$ with $28$ vertices}
\label{fig:28vertex}
\end{figure}
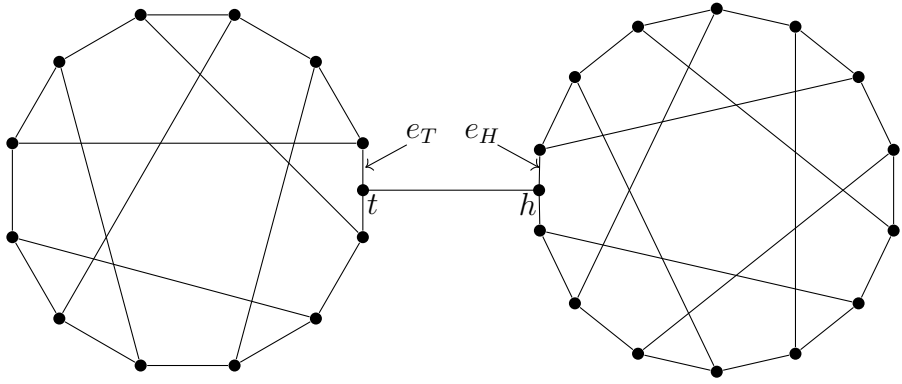 

The following is a simple verification of the det-extremality of $G$. For a graph $F$, let $A(F)$ denote its adjacency matrix. Suppose that $xy$ is a cut edge of  a connected graph $F$, and let $F_1$ and $F_2$ be the two components of $F-xy$ containing $x$ and $y$, respectively. Then
\begin{eqnarray*}
\det A(F)
&=&
\det A(F_1)\det A(F_2)-\det A(F_1-x)\det A(F_2-y),\\
\operatorname{per} A(F)
&=&
\operatorname{per} A(F_1)\operatorname{per} A(F_2)+\operatorname{per} A(F_1-x)\operatorname{per} A(F_2-y).
\end{eqnarray*}
Direct computation using a computer gives 
\begin{eqnarray*}
\det A(\mathbb{T}^*)&=&\operatorname{per}A(\mathbb{T}^*)=128,
\qquad
\det A(\mathbb{T}-e_T)=\operatorname{per}A(\mathbb{T}-e_T)=64. \\
\det A(\HH^*)&=&\operatorname{per}A(\HH^*)=256,
\qquad
\det A(\HH-e_H)=-256,
\qquad
\operatorname{per}A(\HH-e_H)=256.
\end{eqnarray*}
Since $\mathbb{T}^*-t=\mathbb{T}-e_T$, $\HH^*-h=\HH-e_H$, and the edge $th$ is a bridge of $G$, 
we obtain
\[
\det A(G)=128\cdot256-64(-256)=49152,
\qquad
\operatorname{per}A(G)=128\cdot256+64\cdot256=49152.
\]
Hence $G$ is det-extremal.  
We remark that the choice of $e_T$ matters. 
Only two edges of $\mathbb{T}$, forming a single orbit under the automorphism group $\mathrm{Aut}(\mathbb{T})$, yield the above values,
and  
some other choice of $e_T$ may result in a graph that is not
det-extremal.
In contrast, the choice of $e_H$ is irrelevant since the
Heawood graph is edge-transitive.

Let $G$ be a connected bipartite graph with parts $X$ and $Y$.
A \emph{polarity} of $G$ is an automorphism $\tau$ of $G$ with $\tau(X)=Y$ and $\tau^2=\mathrm{id}$, and a vertex $u$ of $G$ is \emph{absolute} for $\tau$ if $u$ and $\tau(u)$ are adjacent in $G$.
Note that a polarity has no fixed vertex by the definition.
 
When $G$ is the incidence graph (also called the \emph{Levi} graph) of a finite projective plane, a polarity of $G$ corresponds naturally to a polarity of the projective plane.
A classical theorem of Baer \cite{Baer} states that a polarity of a finite projective plane of order $q$ has at least $q+1$ absolute points. 
In terms of its Levi graph, this means that the corresponding polarity has at least $q+1$ absolute vertices in each partite set.
In particular, since the Heawood graph $\HH$ is the incidence graph of the Fano
plane ($q=2$), every polarity of $\HH$ has at least three absolute vertices in each
partite set.
 
\begin{observation}\label{lem:polarity}
Let $e=xy$ be an edge of the Heawood graph $\HH$, and let $\tau$ be a polarity of $\HH-e$
with $\tau(x)=y$.
Then $\tau$ has an absolute vertex.
\end{observation}
 
\begin{proof}
Since $\tau(x)=y$ and $\tau^2=\mathrm{id}$, we have $\tau(y)=x$.
Thus $\tau$ maps $e$ to itself, and hence $\tau$ is also a polarity of $\HH$.
Since $\tau$ has at least three absolute vertices in each partite set of $\HH$, let $u_1,u_2,u_3$ be distinct absolute vertices of $\tau$ in the same partite set.
Then the edges $u_1\tau(u_1)$, $u_2\tau(u_2)$, $u_3\tau(u_3)$ of $\HH$ are distinct, and at most one of them is $e$.
Hence at least two of them are edges of $\HH-e$.
In particular, $\tau$ has an absolute vertex in $\HH-e$.
\end{proof}

\begin{proof}[Proof of Theorem~\ref{thm:no22}]
By the first two paragraphs of this subsection  and the graph in Figure~\ref{fig:28vertex}, it is enough to show that  when $G$ is a connected det-extremal cubic non-bipartite graph, $|V(G)|=22$ cannot occur. Suppose to the contrary that $|V(G)|=22$ for a connected det-extremal cubic non-bipartite graph. 
Let $L=G\times K_2$.
Then $L$ is a connected det-extremal cubic bipartite graph of order $44$ by Lemma~\ref{lem:coverprop}(1) and (3).
Let $\sigma$ be the automorphism of $L$ sending $(v,i)$ to $(v,1-i)$ for each $i\in\{0,1\}$.
Then $\sigma$ is a polarity of $L$.
Moreover, $\sigma$ has no absolute vertex, since $G$ has no loop. 
 
As $44$ is not of the form $12k+2$ for a positive integer $k$, Lemma~\ref{lem:Heawood_Property}(2) says that $L$ is not a $k$-vertex-sum of the Heawood graph for any integer $k\ge 1$.
By Theorem~\ref{thm:mccuaig}, $L$ is not $3$-connected.
By Lemma~\ref{lem:parity}(2), the connectivity of $L$ is $2$.
We extract the following claim from the proof of \cite[Theorem~5.1]{FJLS} in the case of order $44$.

\begin{claim}\label{cite:proof:FJLS}
The decomposition cactus $T$ of $L$ consists of three $2$-cycle blocks sharing one common vertex.
The $3$-connected pieces at the three vertices of degree $2$ in $T$ are copies $G_1$, $G_2$, $G_3$ of  the Heawood graph $\HH$, and the $3$-connected piece at the common vertex in $T$ is a $3$-bond whose equivalence class consists of two non-adjacent vertices $a$ and $b$. 
\end{claim}
 
For each $i\in\{1,2,3\}$, let $f_i$ be the marker edge of $G_i$ of Claim~\ref{cite:proof:FJLS}, which is unique by Lemma~\ref{lem:cactus}(1), and let $H_i=G_i-f_i$ as in Definition~\ref{def:decompo}.
By Definition~\ref{def:decompo}, each $2$-cycle block of the cactus gives a $2$-edge cut of $L$ whose two edges join the two ends of $f_i$ to $a$ and $b$, and since the two ends of $f_i$ lie in different parts of $L$ and so do $a$ and $b$, one of the two edges is incident with $a$ and the other with $b$.
Writing $f_i=x_iy_i$ so that $ax_i,by_i\in E(L)$, it follows that $L$ is obtained from the disjoint union of $H_1$, $H_2$ and $H_3$ by adding the two vertices $a$ and $b$ together with the edges $ax_i$ and $by_i$ for $i\in\{1,2,3\}$.

We will show that there exists $i$ such that $\sigma|_{V(H_i)}$ is a polarity of $H_i$ with
$\sigma(x_i)=y_i$.
Every automorphism of $L$ preserves the relation $R$ of Definition~\ref{def:decompo} and therefore permutes its equivalence classes.
Since $\{a,b\}$ is the unique equivalence class with fewer than $14$ vertices, $\sigma$ maps $\{a,b\}$ to itself.
As $\sigma$ has no fixed vertex, we have $\sigma(a)=b$.
It follows that $\sigma$ permutes $V(H_1),V(H_2),V(H_3)$.
Since $\sigma^2=\mathrm{id}$ and $\{\sigma(V(H_1)),\sigma(V(H_2)),\sigma(V(H_3))\}=\{V(H_1),V(H_2),V(H_3)\}$, there exists $i\in\{1,2,3\}$ such that $\sigma(V(H_i))=V(H_i)$. 
Since $x_i$ is the unique neighbor of $a$ in $H_i$ and $y_i$ is the unique neighbor of $b$ in $H_i$, the equality $\sigma(a)=b$ implies $\sigma(x_i)=y_i$.
Therefore $\sigma|_{V(H_i)}$ is a polarity of $H_i$ with $\sigma(x_i)=y_i$.

Since $H_i$ is isomorphic to $\HH-e$ for an edge $e$ of $\HH$, Observation~\ref{lem:polarity} implies that $\sigma|_{V(H_i)}$ has an absolute vertex.
This vertex is also absolute for $\sigma$, contradicting the fact that $\sigma$ has no absolute vertex.
\end{proof}

\section{A bridge to configurations}\label{sec:B:configuration} 
 
A \emph{(symmetric) $k$-configuration} is a $k$-uniform $k$-regular hypergraph in which two distinct vertices lie in at most one common hyperedge; for $k=3$ this is the notion of a $3$-configuration from Section~\ref{sec:intro}.
Its vertices are called \emph{points} and its hyperedges \emph{lines}, and it may be disconnected.
Its incidence graph is also called its \emph{Levi graph}.
The following is folklore.  
 
\begin{observation}\label{lem:config}
Let $G$ be a $k$-regular graph without a $4$-cycle. Then $\ONH(G)$ is a $k$-configuration.
\end{observation}
 
\begin{proof}
The hypergraph $\ONH(G)$ has $|V(G)|$ vertices and as many hyperedges, and is $k$-uniform and $k$-regular.
If two hyperedges $N_G(u),N_G(v)$ with $u\ne v$ met in two points $a\ne b$, then $uavbu$ would be a $4$-cycle of $G$.
Hence two distinct vertices lie in at most one common hyperedge and so $\ONH(G)$ is a $k$-configuration.
\end{proof}
 
Recall from Section~\ref{sec:intro} that a \emph{blocking set} is a set of points meeting every line and containing no line.
Taking a blocking set as one color class and its complement as the other shows that a $k$-configuration has a blocking set if and only if it is $2$-colorable.
So blocking set free means the same as non-$2$-colorable, and we use the two languages interchangeably.
Our main results can therefore be read as statements about blocking set free $3$-configurations.  
 
Two distinct points are \emph{collinear} if they lie on a common line.
A \emph{triangle} is a set of three pairwise collinear points lying on three distinct lines.
By \cite[Proposition~1]{BGPZ}, an incidence structure is a $3$-configuration if and only if its Levi graph is cubic of girth at least $6$.
Thus, a $3$-configuration is triangle-free if and only if its Levi graph has girth at least $8$, since its Levi graph is bipartite.
There is a classical statement by Thomassen~\cite{Thom86} that a connected $3$-configuration is blocking set free if and only if its Levi graph is det-extremal.
Our theorem, Theorem~\ref{thm:thomassen}, is more general, as it speaks of all connected $3$-uniform $3$-regular hypergraphs, not only of those that are $3$-configurations.
As we mentioned in Subsection~\ref{sub:intro:results}, the existence of blocking sets in a $3$-configuration was investigated in the literature, see~\cite{DGropp,HGropp91,Gropp,EGS21}. 
 
We also note that, for a connected det-extremal cubic graph $G$, the hypergraph $\ONH(G)$ is a blocking set free $3$-configuration by Theorem~\ref{thm:detper}, Corollary~\ref{cor:4cycle}, and Observation~\ref{lem:config}. Thus every graph in Theorem~\ref{thm:3connected} or in \cite{AAKSZ} yields a blocking set free $3$-configuration.
By Theorem~\ref{thm:thomassen} and Lemma~\ref{lem:detgirth} we obtain the following.
 
\begin{corollary}\label{cor:triangle}
Every triangle-free $3$-configuration has a blocking set.
\end{corollary}
 
\begin{proof}
Let $\CC$ be a blocking set free $3$-configuration.
It is enough to show that $\CC$ contains a triangle.
Then at least one connected component $\CC'$ of $\CC$ is a blocking set free $3$-configuration.
Let $L'$ be the Levi graph of $\CC'$. 
Then
\begin{align*}
 \CC' \text{ is blocking set free}
 & \Leftrightarrow \text{$\CC'$ is non-$2$-colorable}
 & \text{by the definitions} \\
 & \Leftrightarrow \text{$L'$ is det-extremal}
 & \text{by Theorem~\ref{thm:thomassen}} \\
 & \Rightarrow \text{$L'$ has girth $6$}
 & \text{by Lemma~\ref{lem:detgirth}} \\
 & \Leftrightarrow \text{$\CC'$ has a triangle}
 & \text{by the definitions}.
\end{align*}
Thus, $\CC$ also has a triangle.
\end{proof}
 
In \cite{BGPZ}, the authors summarize the long history of triangle-free $3$-configurations, starting from the Cremona--Richmond configuration, which is the smallest triangle-free $3$-configuration, denoted by $15_3$.
Since then, much further work on its structures has been carried out, 
see \cite{BBP,BGPZ,HGropp91, Gropp,Boben},  and it is known that complete enumeration of triangle-free $3$-configurations remains difficult.
For small orders, computations are also reported in several papers \cite{BBP,BGPZ,AlB10,AlB25}; in particular, it has been verified computationally that no triangle-free $3$-configuration of order at most $24$ is blocking set free.
Corollary~\ref{cor:triangle} proves that all triangle-free $3$-configurations of every order have a blocking set.

\section*{Acknowledgements}

Boram Park was supported by the National Research Foundation of Korea (NRF) grant funded by the Korea government (MSIT) (No.~RS-2025-00523206) and by the New Faculty Startup Fund from Seoul National University.

\bigskip 
\noindent\textbf{Declaration of generative AI use.}
\par\smallskip
\noindent  During the preparation of this manuscript, the authors occasionally used Claude (Anthropic) to suggest improvements to the wording of selected passages. The authors reviewed and edited all such suggestions and take full responsibility for the final manuscript.

\end{document}